\documentclass[11pt,american,english]{article}
\usepackage[T1]{fontenc}
\usepackage[latin9]{inputenc}
\usepackage[a4paper]{geometry}
\usepackage{babel}
\usepackage{amsmath}
\usepackage{amsthm}
\usepackage{amssymb}
\usepackage{setspace}
\usepackage[unicode=true,pdfusetitle,
 bookmarks=true,bookmarksnumbered=true,bookmarksopen=true,bookmarksopenlevel=3,
 breaklinks=false,pdfborder={0 0 1},backref=false,colorlinks=false]
 {hyperref}
 \usepackage{authblk}

\makeatletter
\theoremstyle{plain}
\newtheorem{thm}{\protect\theoremname}
\theoremstyle{remark}
\newtheorem{rem}[thm]{\protect\remarkname}
\theoremstyle{plain}
\newtheorem{prop}[thm]{\protect\propositionname}
\newtheorem{definition}[thm]{\protect\definitionname}
\usepackage{multirow}

\usepackage{arydshln}

\usepackage{lmodern}
\newcommand{\1}{\mbox{1\hspace{-1mm}I}}
\numberwithin{equation}{section}

\usepackage{calc}
\makeatother

\addto\captionsamerican{\renewcommand{\definitionname}{Definition}}
\addto\captionsamerican{\renewcommand{\propositionname}{Proposition}}
\addto\captionsamerican{\renewcommand{\remarkname}{Remark}}
\addto\captionsamerican{\renewcommand{\theoremname}{Theorem}}
\addto\captionsenglish{\renewcommand{\definitionname}{Definition}}
\addto\captionsenglish{\renewcommand{\propositionname}{Proposition}}
\addto\captionsenglish{\renewcommand{\remarkname}{Remark}}
\addto\captionsenglish{\renewcommand{\theoremname}{Theorem}}
\providecommand{\definitionname}{Definition}
\providecommand{\propositionname}{Proposition}
\providecommand{\remarkname}{Remark}
\providecommand{\theoremname}{Theorem}

\begin{document}
\selectlanguage{american}%
\global\long\def\1{\mbox{1\hspace{-1mm}I}}%

\title{Control in Hilbert Space and First Order Mean Field Type Problem}

\author[,1,2]{Alain Bensoussan\footnote{Corresponding author. E-mail: Alain.Bensoussan@utdallas.edu. Research supported by grant
from the National Science Foundation 1905449 and grant from the Hong Kong SAR RGC GRF 11303316.}}
\author[,2]{Henry Hang Cheung\footnote{This work constitutes part of his PhD study at City University of Hong Kong.}}
\author[,3]{Sheung Chi Phillip Yam \footnote{E-mail: scpyam@sta.cuhk.edu.hk. Phillip Yam acknowledges the financial supports from HKGRF-14300717 with the project title ``New kinds of Forward-backward Stochastic Systems with Applications'', HKGRF-14300319 with the project title ``Shape-constrained Inference: Testing
for Monotonicity'', HKGRF-14301321 with the project title ``General Theory for Infinite Dimensional Stochastic Control: Mean Field and Some Classical Problems'', and Germany/Hong Kong Joint Research Scheme Project No. G-HKU701/20 with the project title ``Asymmetry in Dynamically correlated threshold Stochastic volatility model''. He also thanks Columbia University for the kind invitation to be a visiting faculty member in the Department
of Statistics during his sabbatical leave. He also recalled the unforgettable moments and the happiness shared with his beloved
father and used this work in memory of his father's brave battle against liver cancer.}}
\affil[1]{International Center for Decision and Risk Analysis, Jindal School of Management, University of Texas at Dallas}
\affil[2]{School of Data Science, City University of Hong Kong}
\affil[3]{Department of Statistics, The Chinese University of Hong Kong}
\date{\it Dedicated in memory of Mark Davis\\
for his outstanding contribution in control theory.}
\maketitle
\selectlanguage{english}%
\begin{abstract}
We extend the work \cite{bensoussan2019control} by two of the coauthors, which dealt with a deterministic control problem for which the Hilbert space could be generic and investigated a novel form of the `lifting' technique proposed by P. L. Lions. In \cite{bensoussan2019control}, we only showed the local existence and uniqueness of solutions to the FBODEs in the Hilbert space which were associated to the control problems with drift function consisting of the control only. In this article, we establish the global existence and uniqueness of the solutions to the FBODEs in Hilbert space corresponding to control problems with separable drift function which is nonlinear in state and linear in control. We shall also prove the sufficiency of the Pontryagin Maximum Principle and derive the corresponding Bellman equation. Besides, we shall show an analogue in the stationary case. Finally, by using the `lifting' idea as in \cite{stochasticv2,stochasticv1}, we shall apply the result to solve the linear quadratic mean field type control problems, and to show the global existence of the corresponding Bellman equations. 
\end{abstract}
\section{INTRODUCTION}
In recent years, Mean Field Game (MFG) and Mean Field Type Control Theory (MFTCT) are burgeoning. Carmona and Delarue \cite{CarmonaDelaure2} proved the existence of the general forward-backward systems of equations of McKean-Vlasov type using the probabilistic approach, and therefore obtained the classical solution to the master equation arisen from MFG. Their assumptions restricted their application to LQ models only. Cardaliaguet et al. \cite{Lions_Cardaliaguet} proved the existence of the classical solution to the master equation arisen from MFG by PDE techniques and the method of characteristics. To do so, they required the state space to be compact, and the Hamiltonian to be smooth, globally Lipschitz continuous
and to satisfy a certain coercivity condition. Buckdahn et al. \cite{Buckdahn_Peng} adopted
a similar approach to study forward flows, proving that the semigroup of a standard
McKean-Vlasov stochastic differential equation with smooth coefficients is the classical
solution of a particular type of master equation. A crucial assumption was made therein on the smoothness
of the coefficients, which restricted the scope of applications. Gangbo and M\'esz\'aros in \cite{Gangbo_Meszaros} constructed global solutions to the master equation in potential Mean Field Games,
where displacement convexity was used as a substitution for the monotonicity condition. Besides the notion of classical solutions, Mou and Zhang in \cite{Mou_Zhang} gave a notion of weak solution of the master equation arisen from mean field games, using their results of mollifiers on the infinite dimensional space. More results can be found in
the papers of Cosso and Pham \cite{Pham}, Pham and Wei \cite{Pham_Wei} and Djete et al. \cite{Tan}, which concern the Bellman
and Master equations of Mean Field Games and Mean Field Type Control Theory.\\
\hfill\\
By Pontryagin Maximum Principle, MFG and MFTCT are deeply connected to mean field forward backward stochastic differential equations. Pardoux and Tang \cite{Pardoux_Tang}, Antonelli \cite{Antonelli} and Hu and Yong \cite{Hu_Yong} showed the existence and uniqueness of FBSDEs under small time intervals by a fixed point argument. For Markovian FBSDEs, to get rid of the small time issue, Ma et al. \cite{FourStep} employed the Four Step Scheme. They constructed decoupling functions by the use of the classical solutions of quasi-linear PDEs, hence non-degeneracy of the diffusion coefficient and the strong regularity condition on the coefficients were required. Another way to remove time constraints in Markovian FBSDEs was by Delarue \cite{Delaure2002}. Local solutions were patched together by the use of decoupling functions. PDE methods were used to bound the coefficients of the terminal function relative to the initial data in order for the problem to be well-posed. It was later extended to the case of non-Markovian FBSDEs by Zhang in \cite{Zhang}. Moreover, to deal with non-Markovian FBSDEs with arbitrary time length, there was the pure probabilistic method -- method of continuation. It required monotonicity conditions on the coefficients. For seminal works one may consult \cite{Hu_Peng,Peng_Wu,Yong1,Yong2}. With the help of decoupling functions as in \cite{Delaure2002}, but using a BSDE to control the terminal coefficient instead of PDEs, Ma et al. \cite{ZhangMaWuZhang} covered most of the above cases, but in the case of codomain being $\mathbb{R}$. For mean field type FBSDE. A rather general existence result but with a restrictive assumption (boundedness of the coefficients with respect to the state variable) was first done in \cite{CarmonaDelaure} by Carmona and Delarue. Taking advantage of the convexity of the underlying Hamiltonian and applying the continuation method, Carmona and Delarue extended their results in \cite{CarmonaDelaure2}. Bensoussan et al. \cite{BensoussanYamZhang} exploited the condition in \cite{CarmonaDelaure2} and gave weaker conditions for which the results in \cite{CarmonaDelaure2} still hold. By the method of continuation, Ahuja et al. \cite{AhujaRenYang} extended the above result to the FBSDEs which allow coefficients to be functionals of the processes. More details can be found in the monographs \cite{CarmonaDelaureBook,BensoussanYamBook} and \cite{BensoussanFrehseYamMaster,BensoussanFrehseYamInterpretation,YamLQmeanfield}.\\
\hfill\\
We establish the global existence and uniqueness of the solutions to the FBODEs in Hilbert space corresponding to control problems with separable drift function which is nonlinear in state and linear in control. The result can be applied to solve linear quadratic mean field type control problems. We exploit the `lifting to Hilbert space' approach suggested by P. L. Lions in \cite{LionsLecture1,LionsLecture2}, but lift to another Hilbert space instead of the space of random variables. After lifting, the problems are akin to standard control problems, but the drawback is that they are in the infinite dimensional space. By the Pontryagin Maximum Principle, the control problems are reduced to FBODEs in the Hilbert space. In order to accommodate nonlinear settings, we make use of the idea of decoupling. By a Banach fixed point argument, we are able to locally find a decoupling function for the FBODEs. We then derive \textit{a priori} estimates of the decoupling function and extend the solution from local to global as in Delarue \cite{Delaure2002} by the \textit{a priori} estimates. Besides, we also show the analogue in the stationary case. Finally we apply our result to solve linear quadratic mean field type control problems and obtain their corresponding Bellman equations.\\
\hfill\\
The rest of this article is organized as follows. In Section 2, we introduce the model in the Hilbert space. In Section
3, we express the related FBODE and define the decoupling function. \textit{A priori} estimates of the decoupling function are derived in Section 4. In Section 5, we prove the local existence and uniqueness of the FBODE by using a Banach fixed point argument on the function space containing the decoupling function. In Section 6.1, we construct the global solution by our \textit{a priori} estimates. We show the sufficiency of the Maximum Principle in Section 6.2 and write the corresponding Bellman function in Section 6.3. In Section 7, we show the corresponding result in the stationary case. In Section \ref{application_MFTCT}, we apply our result in the Hilbert space as in \cite{stochasticv2}, to solve the optimal control problem, and show the global existence to the corresponding Bellman equation.
\section{THE MODEL}

\subsection{ASSUMPTIONS\label{subsec:ASSUMPTIONS}}

Let $\mathcal{H}$ be a Hilbert space, with scalar product denoted
by $(\cdot,\cdot)$ . We consider a non-linear operator $A(x):\mathcal{H}\mapsto\mathcal{H}$, such that
\begin{equation}
A(0)=0.\label{eq:1-1}
\end{equation}
We assume that $x\mapsto A(x)$ is $C^{1}$ and that $DA(x)=D_{x}A(x)\in\mathcal{L}(\mathcal{H};\mathcal{H})$
satisfies:
\begin{equation}
||DA(x)||\leq\gamma.\label{eq:1-2}
\end{equation}
We use the notation:
\begin{equation}
(DA(x)y,z)=(D_{x}(A(x),z),y).\label{eq:1-3}
\end{equation}
We also assume that $DA(x)$ is differentiable with a second derivative
$D^{2}A(x)\in\mathcal{L}(\mathcal{H};\mathcal{L}(\mathcal{H};\mathcal{H}))$
with the notation:
\begin{equation}
\left\{\begin{aligned}
&(D_{x}(D_{x}(A(x),z),y),w)=(D_{xx}^{2}A(x)(y)w,z),\\
&D_{x}(D_{x}(A(x),z),y)=(D_{xx}^{2}A(x)(y),z).
\end{aligned}\right.\label{eq:1-4}
\end{equation}
We assume the Lipschitz property:
\begin{equation}
||DA(x_{1})-DA(x_{2})||\leq\dfrac{b|x_{1}-x_{2}|}{1+\max(|x_{1}|,|x_{2}|)},\label{eq:1-5}
\end{equation}
 which implies 
\begin{equation}
||D^{2}A(x)||\leq\dfrac{b}{1+|x|}.\label{eq:1-6}
\end{equation}
In the sequel, we shall make restrictions on the size of $b.$ 

We next consider $x\mapsto F(x)$ and $x\mapsto F_{T}(x)$, functionals on $\mathcal{H},$ which are $C^{2},$ with the properties:
\begin{equation}
\left\{\begin{aligned}
&F(0)=0,\,D_{x}F(0)=0,\\
&\nu|\xi|^{2}\leq(D_{xx}^{2}F(x)\xi,\xi)\leq M|\xi|^{2};
\end{aligned}\right.\label{eq:1-7}
\end{equation}
\begin{equation}
\left\{\begin{aligned}
&F_{T}(0)=0,\,D_{x}F{}_{T}(0)=0,\\
&\nu_{T}|\xi|^{2}\leq(D_{xx}^{2}F_{T}(x)\xi,\xi)\leq M_{T}|\xi|^{2},
\end{aligned}\right.\label{eq:170}
\end{equation}
 and $\nu,\nu_{T}>0.$ $\mathcal{H}$ is the state space. In addition, there is a control space $\mathcal{V}$, also a Hilbert space and
a linear bounded operator $B\in\mathcal{L}(\mathcal{V};\mathcal{H})$, an invertible-self adjoint operator on $\mathcal{V}$, denoted
by $N.$ We assume that
\begin{equation}
(BN^{-1}B^{*}\xi,\xi)\geq m|\xi|^{2},\ m>0.\label{eq:1-8}
\end{equation}

\begin{rem}
\label{rem1-1} The assumption (\ref{eq:1-1}) and the first line
assumptions (\ref{eq:1-7}), (\ref{eq:170}) are of course not necessary.
It is just to simplify the calculations. 
\end{rem}

\subsection{THE PROBLEM}

We consider the following control problem. The state evolution is
governed by the differential equation in $\mathcal{H}$:
\begin{equation}
\left\{\begin{aligned}
&\dfrac{dx}{ds}=A(x)+Bv(s),\\
&x(t)=x,
\end{aligned}\right.\label{eq:1-9}
\end{equation}
in which $v(\cdot)$ is in $L^{2}(t,T;\mathcal{V}).$ It is easy to check
that the state $x(\cdot)$ is uniquely defined and belongs to $H^{1}(t,T;\mathcal{H}).$
We define the payoff functional:
\begin{equation}
J_{xt}(v(\cdot)):=\int_{t}^{T}F(x(s))ds+F_{T}(x(T))+\dfrac{1}{2}\int_{t}^{T}(v(s),Nv(s))ds.\label{eq:1-10}
\end{equation}
This functional is continuous and coercive. If $\mathcal{H}$ were
$\mathbb{R}^{n}$, it would be classical that it has a minimum and thus we could write
the necessary conditions of optimality. But the proof does not carry
over to general Hilbert spaces. Moreover, since $A$ is not linear,
we do not have the convexity property, which would guarantee the
existence and uniqueness of a minimum, and thus a solution of the
necessary conditions of optimality. We shall then write the necessary
conditions of optimality and prove directly the existence and uniqueness
of a solution. 

\section{NECESSARY CONDITIONS OF OPTIMALITY}

\subsection{THE SYSTEM }

It is standard to check the following system of forward-backward equations
in $\mathcal{H}$:
\begin{equation}
\left\{\begin{aligned}
&\dfrac{dy}{ds}=A(y)-BN^{-1}B^{*}z(s),\:t<s<T,\\
&-\dfrac{dz}{ds}=(DA(y(s)))^{*}z(s)+DF(y(s)),\\
&y(t)=x,\:z(T)=DF_{T}(y(T)).
\end{aligned}\right.\label{eq:2-1}
\end{equation}
The optimal state is $y(\cdot)$, and $z(\cdot)$ is the adjoint state. The
optimal control is then:
\begin{equation}
u(s)=-N^{-1}B^{*}z(s).\label{eq:2-2}
\end{equation}
The system (\ref{eq:2-1}) expresses the Pontryagin Maximum Principle.
The objective is to study the system of Equations (\ref{eq:2-1}). 

\subsection{DECOUPLING }

We set 
\begin{equation}
z(t)=\Gamma(x,t).\label{eq:2-3}
\end{equation}
It is standard to check that $z(s)=\Gamma(y(s),s)$. So $y(s)$ is
the solution of the differential equation in $\mathcal{H}$:
\begin{equation}
\left\{\begin{aligned}
\dfrac{dy}{ds}&=A(y)-BN^{-1}B^{*}\Gamma(y(s),s),\\
y(t)&=x,
\end{aligned}\right.\label{eq:2-4}
\end{equation}
and $\Gamma(x,s)$ is the solution of the nonlinear partial differential
equation:
\begin{equation}
\left\{\begin{aligned}
-\dfrac{\partial\Gamma}{\partial s}&=D_{x}\Gamma(x)A(x)+(D_{x}A(x))^{*}\Gamma(x)-D_{x}\Gamma(x)BN^{-1}B^{*}\Gamma(x,s)+D_{x}F(x),\\
\Gamma(x,T)&=D_{x}F_{T}(x).
\end{aligned}\right.\label{eq:2-5}
\end{equation}
If $A(x)=Ax$, $F(x)=\dfrac{1}{2}(x,Mx)$ and $F_{T}(x)=\dfrac{1}{2}(x,M_{T}x)$,
then $\Gamma(x,s)=P(s)x,$ and $P(s)$ is solution of the Riccati
equation:
\begin{equation}
\left\{\begin{aligned}
-\dfrac{dP}{ds}&=P(s)A+A^{*}P(s)-P(s)BN^{-1}B^{*}P(s)+M,\\
P(T)&=M_{T}.
\end{aligned}\right.\label{eq:2-6}
\end{equation}

\section{\textit{A PRIORI} ESTIMATES}

\subsection{FIRST ESTIMATE}

We state the first result:
\begin{prop}
\label{prop3-1} We assume (\ref{eq:1-1}), (\ref{eq:1-2}), (\ref{eq:1-5}), (\ref{eq:1-7}), (\ref{eq:170}),
(\ref{eq:1-8}) and 
\begin{equation}
\dfrac{b^{2}}{16}<(m-k)(\nu-k),\:0<k<\min(m,\nu),\label{eq:2-7}
\end{equation}
 then we have the \textit{a priori} estimate:
\begin{equation}
|\Gamma(x,t)|\leq|x|\left(\dfrac{M_{T}^{2}}{\nu_{T}}+\dfrac{\gamma^{2}+M^{2}}{k}(T-t)\right).\label{eq:2-8}
\end{equation}
\end{prop}

\begin{proof}
From the system (\ref{eq:2-1}), we obtain:
\[
\dfrac{d}{ds}(y(s),z(s))=(A(y(s))-BN^{-1}B^{*}z(s),z(s))-\left((DA(y(s)))^{*}z(s)+DF(y(s)),y(s)\right).
\]
Integration yields:
\begin{equation}
\begin{aligned}
&\ (D_{x}F_{T}(y(T)),y(T))+\int_{t}^{T}(BN^{-1}B^{*}z(s),z(s))ds+\int_{t}^{T}(D_{x}F(y(s)),y(s))ds\\
=&\ (x,z(t))+\int_{t}^{T}\left(A(y(s))-DA(y(s))y(s),z(s)\right)ds.
\end{aligned}\label{eq:2-9}
\end{equation}
We note that 
\begin{equation}
|A(x)-DA(x)x|\leq\dfrac{b}{2}|x|;\label{eq:2-10}
\end{equation}
indeed,
\[
A(x)-DA(x)x=\int_{0}^{1}(DA(\theta x)-DA(x))x\,d\theta,
\]
and from the assumption (\ref{eq:1-5}), we get:
\[
|A(x)-DA(x)x|\leq\int_{0}^{1}\dfrac{b|x|^{2}(1-\theta)}{1+|x|}d\theta,
\]
 which implies (\ref{eq:2-10}). Therefore, from (\ref{eq:2-9}), we
obtain, using assumptions:
\[
(x,z(t))\geq\nu_{T}|y(T)|^{2}+m\int_{t}^{T}|z(s)|^{2}ds+\nu\int_{t}^{T}|y(s)|^{2}ds-\dfrac{b}{2}\int_{t}^{T}|y(s)||z(s)|ds.
\]
 Using (\ref{eq:2-7}), we can state:
\begin{equation}
(x,z(t))\geq\nu_{T}|y(T)|^{2}+k\int_{t}^{T}(|y(s)|^{2}+z(s)|^{2})ds.\label{eq:2-11}
\end{equation}
On the other hand, from the second equation (\ref{eq:2-1}), we write:
\[
z(t)=z(T)+\int_{t}^{T}\left((DA(y(s)))^{*}z(s)+DF(y(s))\right)ds,
\]
 hence 
\begin{align*}
(x.z(t))&=(x,DF_{T}(y(T))+\int_{t}^{T}(DA(y(s))x,z(s))ds+\int_{t}^{T}(x,DF(y(s))ds,\\
(x.z(t))&\leq|x||z(t)|\leq|x|(M_{T}|y(T)|+\int_{t}^{T}\gamma|z(s)|ds+\int_{t}^{T}|y(t)|dt)\\
&\leq\dfrac{1}{2}\left(\nu_{T}|y(T)|^{2}+k\int_{t}^{T}(|y(s)|^{2}+z(s)|^{2})ds\right)+\dfrac{|x|^{2}}{2}\left(\dfrac{M_{T}^{2}}{\nu_{T}}+\dfrac{\gamma^{2}+M^{2}}{k}(T-t)\right).
\end{align*}
From this relation and (\ref{eq:2-11}), we get:
\[
\nu_{T}|y(T)|^{2}+k\int_{t}^{T}(|y(s)|^{2}+z(s)|^{2})ds\leq|x|^{2}\left(\dfrac{M_{T}^{2}}{\nu_{T}}+\dfrac{\gamma^{2}+M^{2}}{k}(T-t)\right).
\]
Therefore,
\[
|x||z(t)|\leq|x|^{2}\left(\dfrac{M_{T}^{2}}{\nu_{T}}+\dfrac{\gamma^{2}+M^{2}}{k}(T-t)\right),
\]
 and the result follows. We write 
\begin{equation}
\alpha_{t}=\dfrac{M_{T}^{2}}{\nu_{T}}+\dfrac{\gamma^{2}+M^{2}}{k}(T-t).\label{eq:2-12}
\end{equation}
Note that in the system (\ref{eq:2-1}), we can write 

\begin{equation}
|z(s)|\leq\alpha_{s}|y(s)|.\label{eq:2-13}
\end{equation}
\end{proof}

\subsection{SECOND ESTIMATE\label{subsec:SECOND-ESTIMATE}}

The second estimate concerns the gradient $D_{x}\Gamma(x,t).$ We have
the following result:
\begin{prop}
\label{prop3-2} We make the assumptions of Proposition \ref{prop3-1}
and 
\begin{equation}
\nu-b\alpha_{0}>0,\label{eq:2-14}
\end{equation}
 then we have the \textit{a priori} estimate:
\begin{equation}
||D_{x}\Gamma(x,t)||\leq\dfrac{M_{T}^{2}}{\nu_{T}}+\dfrac{\gamma^{2}}{m}(T-t)+\int_{t}^{T}\dfrac{(M+b\alpha_{s})^{2}}{\nu-b\alpha_{s}}ds.\label{eq:2-15}
\end{equation}
\end{prop}

\begin{proof}
We differentiate the system (\ref{eq:2-1}) with respect to $x.$
We denote 
\begin{equation}
\mathcal{Y}(s)=D_{x}y(s),\,\mathcal{Z}(s)=D_{x}z(s).\label{eq:2-16}
\end{equation}
Differentiating (\ref{eq:2-1}), we can write, by recalling notation (\ref{eq:1-4}):
\begin{equation}
\begin{aligned}
\dfrac{d}{ds}\mathcal{Y}(s)\xi&=D_{x}A(y(s))\mathcal{Y}(s)\xi-BN^{-1}B^{*}\mathcal{Z}(s)\xi-\dfrac{d}{ds}\mathcal{Z}(s)\xi\\
&=(D_{xx}^{2}A(y(s))\mathcal{Y}(s)(\xi),z(s))+(DA(y(s)))^{*}\mathcal{Z}(s)\xi+D_{xx}^{2}F(y(s))\mathcal{Y}(s)\xi,
\end{aligned}\label{eq:2-17}
\end{equation}
\begin{equation}
\mathcal{Y}(t)\xi=\xi,\;\mathcal{Z}(T)\xi=D_{xx}^{2}F_{T}(y(T))\mathcal{Y}(T)\xi.\label{eq:2-170}
\end{equation}
We compute $\dfrac{d}{ds}(\mathcal{Y}(s)\xi,\mathcal{Z}(s)\xi)$ and
then integrate. We obtain that
\begin{equation}
\begin{aligned}
(\mathcal{Z}(t)\xi,\xi)&=(D_{xx}^{2}F_{T}(y(T))\mathcal{Y}(T)\xi,\mathcal{Y}(T)\xi)+\int_{t}^{T}(BN^{-1}B^{*}\mathcal{Z}(s)\xi,\mathcal{Z}(s)\xi)ds\\
&\quad +\int_{t}^{T}(D_{xx}^{2}F(y(s))\mathcal{Y}(s)\xi,\mathcal{Y}(s)\xi)ds+\int_{t}^{T}(D_{xx}^{2}A(y(s))\mathcal{Y}(s)(\xi)\mathcal{Y}(s)\xi,z(s))\\
&\geq\nu_{T}|\mathcal{Y}(T)\xi|^{2}+m\int_{t}^{T}|\mathcal{Z}(s)\xi|^{2}ds+\int_{t}^{T}(\nu-b\alpha_{s})|\mathcal{Y}(s)\xi|^{2}ds.
\end{aligned}\label{eq:3-1}
\end{equation}
 Also, from the second line of (\ref{eq:2-17}),
\begin{equation}
|\mathcal{Z}(t)\xi|\leq M_{T}|\mathcal{Y}(T)\xi|+\int_{t}^{T}(M+b\alpha_{s})|\mathcal{Y}(s)\xi|ds+\gamma\int_{t}^{T}|\mathcal{Z}(s)\xi|ds.\label{eq:3-2}
\end{equation}
Combining (\ref{eq:3-1}) and (\ref{eq:3-2}) as in Proposition \ref{prop3-1},
we conclude that
\[
|\mathcal{Z}(t)\xi|\leq|\xi|\left(\dfrac{M_{T}^{2}}{\nu_{T}}+\dfrac{\gamma^{2}}{m}(T-t)+\int_{t}^{T}\dfrac{(M+b\alpha_{s})^{2}}{\nu-b\alpha_{s}}ds\right).
\]
Since $\mathcal{Z}(t)\xi=$$D_{x}\Gamma(x,t)$, the result (\ref{eq:2-15})
follows immediately. The proof is complete. 
\end{proof}
We shall call 
\begin{equation}
\beta_{t}=\dfrac{M_{T}^{2}}{\nu_{T}}+\dfrac{\gamma^{2}}{m}(T-t)+\int_{t}^{T}\dfrac{(M+b\alpha_{s})^{2}}{\nu-b\alpha_{s}}ds.\label{eq:3-3}
\end{equation}
Since 
\[
\Gamma(x,t)=\int_{0}^{1}D_{x}\Gamma(\theta x,t)x\,d\theta,
\]
 we also have:
\begin{equation}
|\Gamma(x,t)|\leq\beta_{t}|x|,\label{eq:3-4}
\end{equation}
 so in fact,
\begin{equation}
\left\{\begin{aligned}
&|\Gamma(x,t)|\leq\min(\alpha_{t},\beta_{t})|x|,\\
&||D_{x}\Gamma(x,t)||\leq\beta_{t}.
\end{aligned}\right.\label{eq:3-5}
\end{equation}

\section{LOCAL SOLUTION }

\subsection{FIXED POINT APPROACH }

We want to solve (\ref{eq:2-1}) by a fixed point approach. Suppose
we have a function $\lambda(x,t)$ with values in $\mathcal{H}$ such
that:
\begin{equation}
\left\{\begin{aligned}
&|\lambda(x,t)|\leq\mu_{t}|x|,\\
&||D_{x}\lambda(x,t)||\leq\rho_{t},
\end{aligned}\right.\label{eq:4-1}
\end{equation}
 where $\mu_{t}$ and $\rho_{t}$ are bounded functions on $[T-h,T],$
for some convenient $h.$ These functions will be chosen conveniently
in the sequel, with $\mu_{t}<\rho_{t}.$ We then solve 
\begin{equation}
\left\{\begin{aligned}
&\dfrac{d}{ds}y(s)=A(y(s))-BN^{-1}B^{*}\lambda(y(s),s),\\
&y(t)=x.
\end{aligned}\right.\label{eq:4-2}
\end{equation}
This differential equation defines uniquely $y(s),$ thanks to the
assumptions (\ref{eq:4-1}). We then define 
\begin{equation}
\Lambda(x,t):=D_{x}F_{T}(y(T))+\int_{t}^{T}(DA(y(s)))^{*}\lambda(y(s),s)ds+\int_{t}^{T}D_{x}F(y(s))ds.\label{eq:4-3}
\end{equation}
We want to show that $\mu_{t}$ and $\rho_{t}$ can be chosen such that 
\begin{equation}
|\Lambda(x,t)|\leq\mu_{t}|x|,\ ||D_{x}\Lambda(x,t)||\leq\rho_{t},\label{eq:4-4}
\end{equation}
 and that the map $\lambda\mapsto\Lambda$ has a fixed point.
This will be only possible when $t$ remains close to $T,$ namely
$T-h<t<T,$ with $h$ small. 

\subsection{CHOICE OF FUNCTIONS $\mu_{t}$ AND $\rho_{t}$}

From (\ref{eq:4-2}), we obtain:
\[
\dfrac{d}{ds}|y(s)|\leq\left|\dfrac{d}{ds}y(s)\right|\leq(\gamma+||BN^{-1}B^{*}||\mu_{s})|y(s)|,
\]
 which implies 
\begin{equation}
|y(s)|\leq|x|\exp\left(\int_{t}^{s}(\gamma+||BN^{-1}B^{*}||\mu_{\tau})d\tau\right),\label{eq:4-5}
\end{equation}
and thus from (\ref{eq:4-3}) it follows that
\[
|\Lambda(x,t)|\leq M_{T}|y(T)|+\int_{t}^{T}(M+\gamma\mu_{s})|y(s)|ds.
\]
Using (\ref{eq:4-5}), we obtain:
\[
|\Lambda(x,t)|\leq|x|\left(M_{T}\exp\left(\int_{t}^{T}(\gamma+||BN^{-1}B^{*}||\mu_{\tau})d\tau\right)+\int_{t}^{T}(M+\gamma\mu_{s})\exp\left(\int_{t}^{s}(\gamma+||BN^{-1}B^{*}||\mu_{\tau})d\tau\right)ds\right).
\]
To obtain the first inequality (\ref{eq:4-4}), we must choose the function
$\mu_{t}$ such that 
\begin{equation}
\mu_{t}=M_{T}\exp\left(\int_{t}^{T}(\gamma+||BN^{-1}B^{*}||\mu_{\tau})d\tau\right)+\int_{t}^{T}(M+\gamma\mu_{s})\exp\left(\int_{t}^{s}(\gamma+||BN^{-1}B^{*}||\mu_{\tau})d\tau\right)ds.\label{eq:4-6}
\end{equation}
 So $\mu_{t}$ must be solution of the differential equation of Riccati
type:
\begin{equation}
\left\{\begin{aligned}
\dfrac{d}{dt}\mu_{t}&=-||BN^{-1}B^{*}||\mu_{t}^{2}-2\gamma\mu_{t}-M,\\
\mu_{T}&=M_{T}.
\end{aligned}\right.\label{eq:4-60}
\end{equation}
To proceed, we need to assume that
\begin{equation}
\gamma^{2}<M||BN^{-1}B^{*}||,\label{eq:4-7}
\end{equation}
and we define $\mu_{t}$ bt the formula:
\begin{equation}
\arctan\:\dfrac{\mu_{t}||BN^{-1}B^{*}||+\gamma}{\sqrt{M||BN^{-1}B^{*}||-\gamma^{2}}}=\arctan\:\dfrac{M_{T}||BN^{-1}B^{*}||+\gamma}{\sqrt{M||BN^{-1}B^{*}||-\gamma^{2}}}+\left(\sqrt{M||BN^{-1}B^{*}||-\gamma^{2}}\right)(T-t).\label{eq:4-8}
\end{equation}
For $h>0$, define $\theta_{h}$ with 
\begin{equation}
\arctan\:\dfrac{\theta_{h}||BN^{-1}B^{*}||+\gamma}{\sqrt{M||BN^{-1}B^{*}||-\gamma^{2}}}=\arctan\:\dfrac{M_{T}||BN^{-1}B^{*}||+\gamma}{\sqrt{M||BN^{-1}B^{*}||-\gamma^{2}}}+\left(\sqrt{M||BN^{-1}B^{*}||-\gamma^{2}}\right)h.\label{eq:4-80}
\end{equation}
The number $h$ must be small enough to guarantee that
\begin{equation}
\arctan\:\dfrac{M_{T}||BN^{-1}B^{*}||+\gamma}{\sqrt{M||BN^{-1}B^{*}||-\gamma^{2}}}+\left(\sqrt{M||BN^{-1}B^{*}||-\gamma^{2}}\right)h<\dfrac{\pi}{2}.\label{eq:4-9}
\end{equation}
Formula (\ref{eq:4-8}) defines uniquely $\mu_{t}$ for $T-h<t<T.$
It is decreasing in $t$, with $M_{T}<\mu_{t}<\theta_{h}$.

Therefore, for $T-h<t<T,$ we have defined by (\ref{eq:4-3}) a function
$\Lambda(x,t)$ which satisfies the first condition (\ref{eq:4-4}),
with $\mu_{t}$ defined by equation (\ref{eq:4-8}). We turn now to
the definition of $\rho_{t}.$ Define $\mathcal{Y}(s)=D_{x}y(s),$
see (\ref{eq:4-2}). We have:
\begin{equation}
\left\{\begin{aligned}
\dfrac{d}{ds}\mathcal{Y}(s)&=\left(DA(y(s))-BN^{-1}B^{*}D_{x}\lambda(y(s),s)\right)\mathcal{Y}(s),\\
\mathcal{Y}(t)&=I.
\end{aligned}\right.\label{eq:4-10}
\end{equation}
 We obtain, by techniques already used:
\begin{equation}
||\mathcal{Y}(s)||\leq\exp\left(\int_{t}^{s}(\gamma+||BN^{-1}B^{*}||\rho_{\tau})d\tau\right).\label{eq:4-11}
\end{equation}
We then differentiate $\Lambda(x,t)$ in $x,$ see (\ref{eq:4-3}).
We get:
\begin{align*}
D_{x}\Lambda(x,t)&=D_{xx}^{2}F_{T}(y(T))\mathcal{Y}(T)+\int_{t}^{T}(D_{xx}^{2}A(y(s))\mathcal{Y}(s),\lambda(y(s),s))ds,\\
&\quad+\int_{t}^{T}(D_{x}A(y(s)))^{*}D_{x}\lambda(y(s),s)\mathcal{Y}(s)ds+\int_{t}^{T}D_{xx}^{2}F(y(s))\mathcal{Y}(s)ds,
\end{align*}
 and we obtain: 
\[
||D_{x}\Lambda(x,t)||\leq M_{T}||\mathcal{Y}(T)||+\int_{t}^{T}(M+b\mu_{s}+\gamma\rho_{s})||\mathcal{Y}(s)||ds.
\]
 Since $T-h<t<T,$ we can majorize, using also (\ref{eq:4-11}), to
obtain:
\begin{equation}
||D_{x}\Lambda(x,t)||\leq M_{T}\exp\left(\int_{t}^{T}(\gamma+||BN^{-1}B^{*}||\rho_{s})ds\right)+\int_{t}^{T}(M+b\theta_{h}+\gamma\rho_{s})\exp\left(\int_{t}^{s}(\gamma+||BN^{-1}B^{*}||\rho_{\tau})d\tau\right)ds.\label{eq:4-12}
\end{equation}
We are thus led to looking for $\rho_{t}$ solution of 
\begin{equation}
\rho_{t}=M_{T}\exp\left(\int_{t}^{T}(\gamma+||BN^{-1}B^{*}||\rho_{s})ds\right)+\int_{t}^{T}(M+b\theta_{h}+\gamma\rho_{s})\exp\left(\int_{t}^{s}(\gamma+||BN^{-1}B^{*}||\rho_{\tau})d\tau\right)ds.\label{eq:4-13}
\end{equation}
 This equation is similar to the one definning $\mu_{t},$ see (\ref{eq:4-6}),
with the change of $M$ into $M+b\theta_{h}.$ Hence, by analogy with
(\ref{eq:4-8}), we can assert that:
\begin{equation}
\begin{aligned}
\arctan\:\dfrac{\rho_{t}||BN^{-1}B^{*}||+\gamma}{\sqrt{(M+b\theta_{h})||BN^{-1}B^{*}||-\gamma^{2}}}&=\arctan\:\dfrac{M_{T}||BN^{-1}B^{*}||+\gamma}{\sqrt{(M+b\theta_{h})||BN^{-1}B^{*}||-\gamma^{2}}}\\
&\quad+\left(\sqrt{(M+b\theta_{h})||BN^{-1}B^{*}||-\gamma^{2}}\right)(T-t).
\end{aligned}\label{eq:4-14}
\end{equation}
In order to get a bounded solution for $\rho_{t},$ we need that the
right hand side of (\ref{eq:4-14}) be smaller than $\dfrac{\pi}{2}.$
We need to restrict $h$ more than with (\ref{eq:4-9}), namely:
\begin{equation}
\arctan\:\dfrac{M_{T}||BN^{-1}B^{*}||+\gamma}{\sqrt{M||BN^{-1}B^{*}||-\gamma^{2}}}+\left(\sqrt{(M+b\theta_{h})||BN^{-1}B^{*}||-\gamma^{2}}\right)h<\dfrac{\pi}{2}.\label{eq:4-15}
\end{equation}
Then the function $\rho_{t}$ is well defined on $(T-h,T],$ by formula
(\ref{eq:4-14}) and the function $\Lambda(x,t)$ defined by (\ref{eq:4-3}), for $t\in(T-h,T]$ satisfies (\ref{eq:4-4}) if $\lambda(x,t)$
satisfies (\ref{eq:4-1}). We also claim that 
\begin{equation}
\rho_{t}>\mu_{t}.\label{eq:4-16}
\end{equation}
 Indeed, $\rho_{t}$ satisfies the Riccati equation:
\begin{equation}
\left\{\begin{aligned}
\dfrac{d}{dt}\rho_{t}&=-||BN^{-1}B^{*}||\rho_{t}^{2}-2\gamma\rho_{t}-(M+b\theta_{h}),\\
\rho_{T}&=M_{T},
\end{aligned}\right.\label{eq:4-17}
\end{equation}
 and comparing (\ref{eq:4-60}) and (\ref{eq:4-17}), it is standard
to show the property (\ref{eq:4-16}). 

\subsection{CONTRACTION MAPPING }

We define the space of functions  $(x,t)\in\mathcal{H}\times(T-h,T)\mapsto\lambda(x,t)\in\mathcal{H}\times(T-h,T),$
equipped with the norm:
\begin{equation}
||\lambda||_{h}=\sup_{x\in\mathcal{H},t\in(T-h,T)}\dfrac{|\lambda(x,t)|}{|x|}.\label{eq:4-18}
\end{equation}
This space is a Banach space, denoted by $\mathcal{B}_{h\cdot}$. We
next consider the convex closed subset of $\mathcal{B}_{h\cdot}$ of functions
such that:
\begin{equation}
|\lambda(x,t)|\leq\mu_{t}|x|,\:||D_{x}\lambda(x,t)||\leq\rho_{t},\forall t\in(T-h,T],\label{eq:4-19}
\end{equation}
 where $\mu_{t}$ and $\rho_{t}$ are defined by (\ref{eq:4-8}) and
(\ref{eq:4-14}), respectively. The subset (\ref{eq:4-19}) is denoted
by $\mathcal{C}_{h}.$ The map $\lambda\mapsto\Lambda,$ defined
by (\ref{eq:4-2}) and (\ref{eq:4-3}), is defined from $\mathcal{C}_{h}$
to $\mathcal{C}_{h}.$ We want to show that it leads to a contraction. 

Let $\lambda^{1}(x,t),$ $\lambda^{2}(x,t)$ in $\mathcal{C}_{h}$
and the corresponding functions $\Lambda^{1}(x,t),$ $\Lambda^{2}(x,t)$, which also belong to $\mathcal{C}_{h}$. Let $y^{1}(s),y^{2}(s)$
be the solutions of (\ref{eq:4-2}) corresponding to $\lambda^{1},\lambda^{2}.$
We call $\widetilde{y}(s)=y^{1}(s)-y^{2}(s).$ We have:
\[
\left\{\begin{aligned}
&\dfrac{d}{ds}\widetilde{y}(s)=A(y^{1}(s))-A(y^{2}(s))-BN^{-1}B^{*}(\lambda^{1}(y^{1}(s))-\lambda^{2}(y^{2}(s))),\\
&\widetilde{y}(t)=0,
\end{aligned}\right.
\]
hence 
\[
\dfrac{d}{ds}|\widetilde{y}(s)|\leq\gamma|\widetilde{y}(s)|+||BN^{-1}B^{*}||\,|\lambda^{1}(y^{1}(s))-\lambda^{2}(y^{2}(s))|.
\]
Next,
\begin{align*}
|\lambda^{1}(y^{1}(s))-\lambda^{2}(y^{2}(s))|&\leq|\lambda^{1}(y^{1}(s))-\lambda^{1}(y^{2}(s))|+|\lambda^{1}(y^{2}(s))-\lambda^{2}(y^{2}(s))|\\
&\leq\rho_{s}|\widetilde{y}(s)||+||\lambda^{1}-\lambda^{2}||_{h}|x|\exp\left(\int_{t}^{s}(\gamma+||BN^{-1}B^{*}||\mu_{\tau})d\tau\right).
\end{align*}
 Therefore,
\[
\dfrac{d}{ds}|\widetilde{y}(s)|\leq(\gamma+||BN^{-1}B^{*}||\rho_{s})|\widetilde{y}(s)|+||BN^{-1}B^{*}||\,|x|\,||\lambda^{1}-\lambda^{2}||_{h}\exp\left(\int_{t}^{s}(\gamma+||BN^{-1}B^{*}||\mu_{\tau})d\tau\right).
\]
 We obtain that
\[
|\widetilde{y}(s)|\exp\left(-\int_{t}^{s}(\gamma+||BN^{-1}B^{*}||\rho_{\tau})d\tau\right)\leq||BN^{-1}B^{*}||\,|x|\,||\lambda^{1}-\lambda^{2}||_{h}\int_{t}^{s}\exp\left(-||BN^{-1}B^{*}||\int_{t}^{\tau}(\rho_{\theta}-\mu_{\theta})d\theta\right)d\tau,
\]
 which implies:
\begin{equation}
|\widetilde{y}(s)|\leq h||BN^{-1}B^{*}||\,|x|\,||\lambda^{1}-\lambda^{2}||_{h}\exp\left(\int_{t}^{s}(\gamma+||BN^{-1}B^{*}||\rho_{\tau})d\tau\right).\label{eq:4-20}
\end{equation}
We next have from the definition of the map $\Lambda(x,t)$ that:
\begin{equation}
\begin{aligned}
\Lambda^{1}(x,t)-\Lambda^{2}(x,t)&=DF_{T}(y^{1}(T))-DF_{T}(y^{2}(T))+\int_{t}^{T}\left(DA^{*}(y^{1}(s))\lambda^{1}(y^{1}(s))-DA^{*}(y^{2}(s))\lambda^{2}(y^{2}(s))\right)ds\\
&\quad+\int_{t}^{T}(DF(y^{1}(s))-DF(y^{2}(s)))ds.
\end{aligned}\label{eq:4-21}
\end{equation}
 We have:
\[
|DA^{*}(y^{1}(s))\lambda^{1}(y^{1}(s))-DA^{*}(y^{2}(s))\lambda^{2}(y^{2}(s))|\leq(b\theta_{h}+\gamma\rho_{s})|\widetilde{y}(s)|+\gamma|x|\,||\lambda^{1}-\lambda^{2}||_{h}\exp\left(\int_{t}^{s}(\gamma+||BN^{-1}B^{*}||\mu_{\tau})d\tau\right).
\]
 So, from (\ref{eq:4-21}), we obtain:
\begin{equation}
\begin{aligned}
|\Lambda^{1}(x,t)-\Lambda^{2}(x,t)|&\leq M_{T}|\widetilde{y}(T)|+\int_{t}^{T}(M+b\theta_{h}+\gamma\rho_{s})|\widetilde{y}(s)|ds\\
&\quad+\gamma|x|\,||\lambda^{1}-\lambda^{2}||_{h}\int_{t}^{T}\exp\left(\int_{t}^{s}(\gamma+||BN^{-1}B^{*}||\mu_{\tau})d\tau\right)ds,
\end{aligned}\label{eq:4-22}
\end{equation}
 and from (\ref{eq:4-20}):
\begin{align*}
|\Lambda^{1}(x,t)-\Lambda^{2}(x,t)|&\leq|x|\,|\lambda^{1}-\lambda^{2}||_{h}h\times\left[||BN^{-1}B^{*}||\left(M_{T}\exp\left(\int_{t}^{T}(\gamma+||BN^{-1}B^{*}||\rho_{\tau})d\tau\right)\right.\right.\\
&\quad\left.\left.+\int_{t}^{T}(M+b\theta_{h}+\gamma\rho_{s})\left(\int_{t}^{s}(\gamma+||BN^{-1}B^{*}||\rho_{\tau})d\tau\right)ds\right)\right]\\
&\quad+\gamma|x|\,|\lambda^{1}-\lambda^{2}||_{h}\int_{t}^{T}\exp\left(\int_{t}^{s}(\gamma+||BN^{-1}B^{*}||\mu_{\tau})d\tau\right)ds,
\end{align*}
then from the definition of $\rho_{t}$ (see (\ref{eq:4-13})), we
obtain:
\begin{equation}
|\Lambda^{1}(x,t)-\Lambda^{2}(x,t)|\leq|x|\,|\lambda^{1}-\lambda^{2}||_{h}h\,\left(\rho_{t}||BN^{-1}B^{*}||+\gamma\exp\left(\int_{T-h}^{T}(\gamma+||BN^{-1}B^{*}||\mu_{\tau})d\tau\right)\right).\label{eq:4-23}
\end{equation}
Similarly to the definition of $\theta_{h}$ (see (\ref{eq:4-80})), we define the quantity $\sigma_{h}$ by the formula:
\begin{equation}
\arctan\:\dfrac{\sigma_{h}||BN^{-1}B^{*}||+\gamma}{\sqrt{(M+b\theta_{h})||BN^{-1}B^{*}||-\gamma^{2}}}=\arctan\:\dfrac{M_{T}||BN^{-1}B^{*}||+\gamma}{\sqrt{(M+b\theta_{h})||BN^{-1}B^{*}||-\gamma^{2}}}+\left(\sqrt{(M+b\theta_{h})||BN^{-1}B^{*}||-\gamma^{2}}\right)h.\label{eq:4-24}
\end{equation}
 From (\ref{eq:4-14}), we see that $M_{T}<\rho_{t}<\sigma_{h}.$ Therefore
from (\ref{eq:4-23}),
\begin{equation}
||\Lambda^{1}-\Lambda^{2}||_{h}\leq|\lambda^{1}-\lambda^{2}||_{h}\,h\left(\sigma_{h}||BN^{-1}B^{*}||+\gamma\exp\left(h(\gamma+||BN^{-1}B^{*}||\theta_{h})\right)\right).\label{eq:4-240}
\end{equation}
Using the fact that $\theta_{h}\rightarrow M_{T}$ as $h\rightarrow0,$
equation (\ref{eq:4-24}) shows that $\sigma_{h}\rightarrow M_{T}$
as $h\rightarrow0.$ We deduce that:
\begin{equation}
h\left(\sigma_{h}||BN^{-1}B^{*}||+\gamma\exp\left(h(\gamma+||BN^{-1}B^{*}||\theta_{h})\right)\right)\rightarrow0,\text{ as }h\rightarrow0.\label{eq:4-25}
\end{equation}
We can restrict $h$ such that 
\begin{equation}
h\left(\sigma_{h}||BN^{-1}B^{*}||+\gamma\exp\left(h(\gamma+||BN^{-1}B^{*}||\theta_{h})\right)\right)<1,\label{eq:4-26}
\end{equation}
and thus for $h$ sufficiently small, the map $\lambda\mapsto\Lambda$
is paradoxical and leads to a contradiction. We can summarize the results in the following theorem:
\begin{thm}
\label{theo4-1} We assume $(\ref{eq:4-7}).$ We choose $h$ small
enough to satisfy conditions (\ref{eq:4-9}), (\ref{eq:4-15}), (\ref{eq:4-26}).
For $T-h<t<T$, there exists one and only one solution of the system
of forward-backward equations (\ref{eq:2-1}). We have also one and
only one solution of equation (\ref{eq:2-5}) on the same interval. 
\end{thm}

\section{GLOBAL SOLUTION }

\subsection{STATEMENT OF RESULTS}

We have proven in Theorem \ref{theo4-1} the existence and uniqueness
of a local solution of the system (\ref{eq:2-1}). We want to state
that this solution is global, under the assumptions of Proposition
\ref{prop3-2}.
\begin{thm}
\label{theo5-1}We make the assumptions of Proposition \ref{prop3-2}
and (\ref{eq:4-7}). The local solution defined in Theorem \ref{theo4-1}
can be extented. Thus there exists one and only one solution of the
system (\ref{eq:2-1}) on any finite interval $[0,T]$, and there exists
one and only one solution of equation (\ref{eq:2-5}) on any finite
interval $[0,T].$
\end{thm}

\begin{proof}
Defining by $\Gamma(x,t)$ the fixed point obtained in Theorem \ref{theo4-1},
it is the unique solution of the paraboloic equation: 
\begin{equation}
\left\{\begin{aligned}
-\dfrac{\partial\Gamma}{\partial t}&=D_{x}\Gamma(x)A(x)+(D_{x}A(x))^{*}\Gamma(x)-D_{x}\Gamma(x)BN^{-1}B^{*}\Gamma(x,s)+D_{x}F(x),\:T-h<t<T,\\
\Gamma(x,T)&=D_{x}F_{T}(x),
\end{aligned}\right.\label{eq:5-1}
\end{equation}
with $h$ restricted as stated in Theorem \ref{theo4-1}. We also
have the estimates: 
\begin{equation}
\left\{\begin{aligned}
|\Gamma(x,t)|&\leq\min(\alpha_{t},\beta_{t})|x|,\\
||D_{x}\Gamma(x,t)||&\leq\beta_{t},
\end{aligned}\right.\label{eq:5-2}
\end{equation}
with 
\begin{equation}
\left\{\begin{aligned}
\alpha_{t}&=\dfrac{M_{T}^{2}}{\nu_{T}}+\dfrac{\gamma^{2}+M^{2}}{k}(T-t),\\
\beta_{t}&=\dfrac{M_{T}^{2}}{\nu_{T}}+\dfrac{\gamma^{2}}{m}(T-t)+\int_{t}^{T}\dfrac{(M+b\alpha_{s})^{2}}{\nu-b\alpha_{s}}ds.
\end{aligned}\right.\label{eq:5-3}
\end{equation}
These estimates follow from the \textit{a priori} estimates stated in Proposition
\ref{prop3-1} and \ref{prop3-2}. They do not depend on $h.$ Now we
want to extend (\ref{eq:5-1}) for $t<T-h$. To avoid confusion,
we define 
\begin{equation}
U_{T-h}(x):=\Gamma(x,T-h).\label{eq:5-4}
\end{equation}
We set $M_{T-h}=\beta_{0}.$ We can then state: 
\begin{equation}
\left\{\begin{aligned}
|U_{T-h}(x)|&\leq M_{T-h}|x|,\\
||D_{x}U_{T-h}(x)||&\leq M_{T-h},
\end{aligned}\right.\label{eq:5-5}
\end{equation}
and we consider the parabolic equation:
\begin{equation}
\left\{\begin{aligned}
-\dfrac{\partial\Gamma}{\partial t}&=D_{x}\Gamma(x)A(x)+(D_{x}A(x))^{*}\Gamma(x)-D_{x}\Gamma(x)BN^{-1}B^{*}\Gamma(x,s)+D_{x}F(x),\:t<T-h,\\
\Gamma(x,T-h)&=U_{T-h}(x).
\end{aligned}\right.\label{eq:5-6}
\end{equation}
 We associate to this equation the system: 
\begin{equation}
\left\{\begin{aligned}
\dfrac{dy}{ds}&=A(y)-BN^{-1}B^{*}z(s),\:t<s<T-h,\\
-\dfrac{dz}{ds}&=(DA(y(s)))^{*}z(s)+DF(y(s)),\\
y(t)&=x,\:z(T-h)=U_{T-h}(y(T-h)).
\end{aligned}\right.\label{eq:5-7}
\end{equation}
Proceeding like in Theorem \ref{theo4-1}, we can solve this system
on an interval $[T-h-l,T-h],$ for a sufficiently small $l\not=h.$
The difference is due to the fact that $M_{T-h}\not=M_{T}.$ So in
(\ref{eq:5-1}), we can replace $T-h$ by $T-h-l.$ This time the estimates
on $\Gamma(x,T-h-l)$ and $D_{x}\Gamma(x,T-h-l)$ are identical to
those of $\Gamma(x,T-h)$ and $D_{x}\Gamma(x,T-h)$, thanks to the
\textit{a priori} estimates. So the intervals we can extend further will have
the same length. Clearly, this implies that we can extend (\ref{eq:5-1})
up to $t=0.$ So, we obtain the global existence and uniqueness of
equation (\ref{eq:2-5}) on $[0,T].$ The proof is complete.
\end{proof}

\subsection{OPTIMAL CONTROL}

In Theorem \ref{theo5-1}, we have obtained the existence and uniqueness
of the solution of the pair $(y(s),z(s))$ of the system (\ref{eq:2-1}),
for any $t\in[0,T].$ We want now to check that the control $u(s)$
defined by (\ref{eq:2-2}) is solution of the control problem (\ref{eq:1-9}),
(\ref{eq:1-10}), and that the optimal control is unique.
\begin{thm}
\label{theo5-2} Under the assumptions of Theorem \ref{theo5-1},
the control $u(\cdot)$ defined by (\ref{eq:2-2}) is the unique optimal
control for the problem (\ref{eq:1-9}), (\ref{eq:1-10}). 
\end{thm}

\begin{proof}
Let $v(\cdot)$ be another control. We shall prove that 
\begin{equation}
J(u(\cdot)+v(\cdot))\geq J(u(\cdot)),\label{eq:5-8}
\end{equation}
which will prove the optimality of $u(\cdot).$ We define by $y_{v}(\cdot)$
the state corresponding to the control $u(\cdot)+v(\cdot).$ It is the solution
of 
\begin{equation}
\left\{\begin{aligned}
\dfrac{d}{ds}y_{v}(s)&=A(y_{v}(s))+B(u(s)+v(s)),\\
y_{v}(t)&=x,
\end{aligned}\right.\label{eq:5-9}
\end{equation}
and we have:
\[
J(u(\cdot)+v(\cdot))=\int_{t}^{T}F(y_{v}(s))ds+F_{T}(y_{v}(T))+\dfrac{1}{2}\int_{t}^{T}(u(s)+v(s),N(u(s)+v(s)))ds,
\]
 and 
\begin{align*}
&\ J(u(\cdot)+v(\cdot))-J(u(\cdot))\\
=&\ \int_{t}^{T}(F(y_{v}(s))-F(y(s)))ds+F_{T}(y_{v}(T))-F_{T}(y(T))+\dfrac{1}{2}\int_{t}^{T}(v(s),Nv(s))ds+\int_{t}^{T}(Nu(s),v(s))ds.
\end{align*}
 We denote $\tilde{y}_{v}(s):=y_{v}(s)-y(s).$ It satisfies:
\begin{equation}
\left\{\begin{aligned}
\dfrac{d}{ds}\tilde{y}_{v}(s)&=A(y_{v}(s))-A(y(s))+Bv(s),\\
\tilde{y}_{v}(t)&=0.
\end{aligned}\right.\label{eq:5-10}
\end{equation}
 Then,
\begin{align*}
J(u(\cdot)+v(\cdot))-J(u(\cdot))&=\int_{t}^{T}(D_{x}F(y(s)),\tilde{y}_{v}(s))ds+\int_{t}^{T}\int_{0}^{1}\int_{0}^{1}\theta\left(D_{xx}^{2}F(y(s)+\lambda\theta\tilde{y}_{v}(s))\tilde{y}_{v}(s),\tilde{y}_{v}(s)\right)dsd\lambda d\theta\\
&\quad+(D_{x}F_{T}(y(T)),\tilde{y}_{v}(T))+\int_{0}^{1}\int_{0}^{1}\theta\left(D_{xx}^{2}F_{T}(y(T)+\lambda\theta\tilde{y}_{v}(T))\tilde{y}_{v}(T),\tilde{y}_{v}(T)\right)d\lambda d\theta\\
&\quad+\dfrac{1}{2}\int_{t}^{T}(v(s),Nv(s))ds-\int_{t}^{T}(z(s),Bv(s))ds.
\end{align*}
From the assumptions (\ref{eq:1-7}), we can write:
\begin{align*}
J(u(\cdot)+v(\cdot))-J(u(\cdot))&\geq\int_{t}^{T}\left(-\dfrac{d}{ds}z(s)-DA^{*}(y(s))z(s),\tilde{y}_{v}(s)\right)ds+\dfrac{\nu}{2}\int_{t}^{T}|\tilde{y}_{v}(s)|^{2}ds+(z(T),\tilde{y}_{v}(T))\\
&\quad+\dfrac{\nu_{T}}{2}|\tilde{y}_{v}(T)|^{2}+\dfrac{1}{2}\int_{t}^{T}(v(s),Nv(s))ds-\int_{t}^{T}(z(s),\dfrac{d}{ds}\tilde{y}_{v}(s)-(A(y_{v}(s))-A(y(s))))ds,
\end{align*}
 which reduces to:
\begin{equation}
\begin{aligned}
J(u(\cdot)+v(\cdot))-J(u(\cdot))&\geq\dfrac{\nu}{2}\int_{t}^{T}|\tilde{y}_{v}(s)|^{2}ds+\dfrac{\nu_{T}}{2}|\tilde{y}_{v}(T)|^{2}+\dfrac{1}{2}\int_{t}^{T}(v(s),Nv(s))ds\\
&\quad+\int_{t}^{T}\left(z(s),A(y_{v}(s))-A(y(s))-DA(y(s))\tilde{y}_{v}(s)\right)ds.
\end{aligned}\label{eq:5-11}
\end{equation}
Note that
\[
\left|\left(z(s),A(y_{v}(s))-A(y(s))-DA(y(s))\tilde{y}_{v}(s)\right)\right|\leq\dfrac{b|z(s)||\tilde{y}_{v}(s)|^{2}}{2(1+|y(s)|)}\leq\dfrac{b\alpha_s}{2}|\tilde{y}_{v}(s)|^{2}.
\]
 Finally, we can state that
\begin{equation}
J(u(\cdot)+v(\cdot))-J(u(\cdot))\geq\dfrac{1}{2}\int_{t}^{T}(\nu-b\alpha_s)|\tilde{y}_{v}(s)|^{2}ds+\dfrac{\nu_{T}}{2}|\tilde{y}_{v}(T)|^{2}\dfrac{1}{2}\int_{t}^{T}(v(s),Nv(s))ds.\label{eq:5-12}
\end{equation}
Thanks to the assumption (\ref{eq:2-14}), the right hand side of (\ref{eq:5-12})
is positive, which proves (\ref{eq:5-8}) and completes the proof
of the result.  
\end{proof}

\subsection{BELLMAN EQUATION }

We have proven, under the assumptions of Theorem \ref{theo5-1}, that
the control problem (\ref{eq:1-9}), (\ref{eq:1-10}) has a unique
solution $u(\cdot).$ Defining the value function 
\begin{equation}
V(x,t):=\inf_{v(\cdot)}J_{xt}(v(\cdot))=J_{xt}(u(\cdot)),\label{eq:5-13}
\end{equation}
we can state that:
\begin{equation}
V(x,t)=\int_{t}^{T}F(y(s))ds+F_{T}(y(T))+\dfrac{1}{2}\int_{t}^{T}(BN^{-1}B^{*}\Gamma(y(s),s),\Gamma(y(s),s))ds,\label{eq:5-14}
\end{equation}
with
\begin{equation}
\left\{\begin{aligned}
\dfrac{d}{ds}y(s)&=A(y(s))-BN^{-1}B^{*}\Gamma(y(s),s),\\
y(t)&=x.
\end{aligned}\right.\label{eq:5-15}
\end{equation}
 We first have: 
\begin{prop}
\label{prop5-1}We have the following property: 
\begin{equation}
\Gamma(x,t)=D_{x}V(x,t).\label{eq:5-16}
\end{equation}
\end{prop}

\begin{proof}
Since the minimum of $J_{xt}(v(\cdot))$ is attained in the unique value
$u(\cdot),$ we can rely on the envelope theorem to claim that:
\begin{equation}
(D_{x}V(x,t),\xi)=\int_{t}^{T}(D_{x}F(y(s)),\mathcal{X}(s)\xi)ds+(D_{x}F_{T}(y(T)),\mathcal{X}(T)\xi),\label{eq:5-17}
\end{equation}
 in which $\mathcal{X}(s)$ is the solution of 
\[
\left\{\begin{aligned}
\dfrac{d}{ds}\mathcal{X}(s)&=D_{x}A(y(s))\mathcal{X}(s),\\
\mathcal{X}(t)&=I.
\end{aligned}\right.
\]
Recalling the equation (\ref{eq:2-1}) for $z(s)$ and performing
integration by parts in (\ref{eq:5-17}), the result $(D_{x}V(x,t),\xi)=(\Gamma(x,t),\xi)$
is easily obtained. This proves the result (\ref{eq:5-16}). 
\end{proof}
We can then obtain the Bellman equation for the value function $V(x,t)$.
\begin{thm}
\label{theo5-3}We make the assumptions of Theorem \ref{theo5-1}.
The function $V(x,t)$ is the unique solution of 
\begin{equation}
\left\{\begin{aligned}
&-\dfrac{\partial V}{\partial t}-(D_{x}V,A(x))+\dfrac{1}{2}(D_{x}V,BN^{-1}B^{*}D_{x}V)=F(x),\\
&V(x,T)=F_{T}(x).
\end{aligned}\right.\label{eq:5-18}
\end{equation} 
\end{thm}

\begin{proof}
We know that $V(x,t)$ is G\^{a}teaux differentiable in $x,$ with the derivative
$\Gamma(x,t).$ From (\ref{eq:2-1}), $\Gamma(x,t)$ is continuous
in $t$. From equation (\ref{eq:5-14}), we can write:
\begin{equation}
V(x,t)-V(x,t+\epsilon)=\int_{t}^{t+\epsilon}F(y(s))ds+\dfrac{1}{2}\int_{t}^{t+\epsilon}(BN^{-1}B^{*}\Gamma(y(s),s),\Gamma(y(s),s))ds+V(y(\epsilon),t+\epsilon)-V(x,t+\epsilon).
\label{eq:5-19}
\end{equation}
We then have: 
\begin{small}
\begin{equation}
\begin{aligned}
&\ V(y(\epsilon),t+\epsilon)-V(x,t+\epsilon)\\
=&\ V\left(x+\int_{t}^{t+\epsilon}A(y(s))ds-\int_{t}^{t+\epsilon}BN^{-1}B^{*}\Gamma(y(s),s)ds,t+\epsilon\right)-V(x,t+\epsilon)\\
=&\ \left(\Gamma(x,t+\epsilon),\int_{t}^{t+\epsilon}A(y(s))ds-\int_{t}^{t+\epsilon}BN^{-1}B^{*}\Gamma(y(s),s)ds\right)\\
&\ +\int_{0}^{1}\left(\Gamma(x+\theta\int_{t}^{t+\epsilon}(A(y(s))-BN^{-1}B^{*}\Gamma(y(s)))ds,t+\epsilon)-\Gamma(x,t+\epsilon),\int_{t}^{t+\epsilon}(A(y(s))-BN^{-1}B^{*}\Gamma(y(s)))ds\right)d\theta.
\end{aligned}\label{eq:5-20}
\end{equation}
\end{small}%
 Using the fact that $\Gamma(x,t)$ is uniformly Lipschitz in $x$ and
continuous in $t,$ we obtain easily from (\ref{eq:5-20}) that:
\[
\dfrac{V(y(\epsilon),t+\epsilon)-V(x,t+\epsilon)}{\epsilon}\rightarrow(\Gamma(x,t),A(x)-BN^{-1}B^{*}\Gamma(x,t)).
\]
Then, dividing (\ref{eq:5-19}) by $\epsilon$ and letting $\epsilon$
tend to $0,$ we obtain the PDE (\ref{eq:5-18}), recalling (\ref{eq:5-16}). The intial condition in (\ref{eq:5-18}) is trivial. If we take
the gradient in $x$ of (\ref{eq:5-18}), we recognize equation (\ref{eq:2-5}).
Since this equation has a unique solution, the solution of (\ref{eq:5-18})
is also unique (easy checking). This completes the proof.  
\end{proof}

\section{APPLICATION TO MEAN FIELD TYPE CONTROL THEORY}
\label{application_MFTCT}
\subsection{WASSERSTEIN SPACE}
Denote by $\mathcal{P}_2(\mathbb{R}^n)$ the Wasserstein space of Borel probability measures $m$ on $\mathbb{R}^n$ such that $\int_{\mathbb{R}^n}|x|^2dm(x)<\infty$, with the metric
\begin{align}
    W_2(\mu,\nu) = \sqrt{\inf\Bigg\{\int|x-y|^2d\pi(x,y):\pi\in\Pi(\mu,\nu)\Bigg\}},
\end{align}
where $\Pi(\mu,\nu)$ is the space of all Borel probability measures on $\mathbb{R}^n\times\mathbb{R}^n$ whose first and second marginals are $\mu$ and $\nu$ respectively.
\subsection{FUNCTIONAL DERIVATIVES}
Let $F$ be a functional on $\mathcal{P}_2(\mathbb{R}^n)$. We recall the idea of the functional derivative here.
\begin{definition}
$F$ is said to have a functional derivative if there exists a continous function $\dfrac{dF}{dm}:\mathcal{P}_2(\mathbb{R}^n)\times \mathbb{R}^n \to \mathbb{R}$, such that for some $c:\mathcal{P}_2(\mathbb{R}^n)\to [0,\infty)$ which is bounded on bounded subsets, we have 
\begin{align}
    \Bigg|\dfrac{dF}{dm}(m,x)\Bigg|\leq c(m)(1+|x|^2)
\end{align}
and
\begin{align}
    F(m') - F(m) = \int_0^1\int_{\mathbb{R}^n}\dfrac{dF}{dm}(m+\theta(m'-m))(x)d(m'-m)(x)d\theta.
\end{align}
We require also $\int_{\mathbb{R}^n} \frac{dF}{dm}(m,x)dm(x)=0$ as it is unique up to a constant by definition.
\end{definition}
\begin{definition}
$F$ is said to have a second order functional derivative if there exists a continuous function $\dfrac{d^2F}{dm^2}:\mathcal{P}_2\times \mathbb{R}^n\times \mathbb{R}^n \to \mathbb{R}$ such that, for some $c:\mathcal{P}_2(\mathbb{R}^n)\to [0,\infty)$ which is bounded on bounded subsets, we have
\begin{align}
    \Bigg|\dfrac{d^2F}{dm^2}(m,x,\tilde{x}) \Bigg|\leq c(m)(1+|x|^2 + |\tilde{x}|^2)
\end{align}
and
\begin{align}
    F(m')-F(m) =& \int_{\mathbb{R}^n}\dfrac{dF}{dm}(m)(x)d(m'-m)(x)\\
    &+\int_0^1\int_0^1\theta \dfrac{d^2F}{dm^2} (m+\lambda\theta(m'-m))(x,\tilde{x})d(m'-m)(x)d(m'-m)(\tilde{x})d\lambda d\theta.\nonumber
\end{align}
Again we require also $\int \frac{d^2F}{dm}(m,x,\tilde{x})dm(\tilde{x}) = 0$ $\forall x$ and $\int \frac{d^2F}{dm}(m,x,\tilde{x})dm(x) = 0$ $\forall \tilde{x}$ as it is unique up to a constant. Note also that 
\begin{align}
    \dfrac{d^2F}{dm^2}(m)(x,\tilde{x})=\dfrac{d^2 F}{dm^2}(m)(\tilde{x},x).
\end{align}
\end{definition}
\noindent We write $D\frac{dF}{dm}(m)(x)$ to mean differentiating with respect to $x$, and $D_1\frac{d^2F}{dm^2}(m)(x_1,x_2)$ and $D_2\frac{d^2F}{dm^2}(m)(x_1,x_2)$ to denote partial differentiation with respect to $x_1$ and $x_2$, respectively. 
\subsection{MEAN FIELD TYPE CONTROL PROBLEMS}
\label{mean_field_type_control}
We introduce the setting of a mean-field type control problem.
Consider real-valued functions $f(x,m)$ and $h(x,m)$ defined on $\mathbb{R}^n\times \mathcal{P}_2(\mathbb{R}^n)$. We define
\begin{align*}
    F(m)&:= \int_{\mathbb{R}^n}f(x,m)dm(x),\\
    F_T(m)&:= \int_{\mathbb{R}^n}h(x,m)dm(x).
\end{align*}
Fix a $m\in\mathcal{P}_2(\mathbb{R}^n)$. Let $A, B:\mathbb{R}^n\to \mathbb{R}^n$ be matrices, and $N:\mathbb{R}^n\to\mathbb{R}^n$ be a self-adjoint invertible matrix. We make the following assumptions on $f$, $h$, $B$, $N$, $A$. We assume that
\begin{enumerate}
    \item [(A1)] $\forall x\in\mathbb{R}^n$,
    \begin{align}
        BN^{-1}B^*x\cdot x \geq m|x|^2, m>0.
    \end{align}
    \item [(A2)] $f$ is regular enough such that the following is justifiable. $\forall y\in\mathbb{R}^n$,
    \begin{align}
        \nu |y|^2 \leq &\dfrac{\partial^2 f}{\partial x^2}(x,m)y\cdot y \leq M |y|^2,\\
        \nu |y|^2 \leq D_{\xi}^2 &\dfrac{\partial f}{\partial m}(x,m)(\xi)y\cdot y \leq M|y|^2,\\
        D_{\xi}&\dfrac{\partial ^2 f}{\partial x\partial m}(x,m)(\xi) = 0,\\
        D_{\xi_1}D_{\xi_2}&\dfrac{\partial^2 f}{\partial m^2}(x,m)(\xi_1,\xi_2) = 0.
    \end{align}
    \item [(A3)] $h$ is regular enough such that the following is justifiable. $\forall y\in\mathbb{R}^n$,
    \begin{align}
        \nu_T |y|^2 \leq &\dfrac{\partial^2 h}{\partial x^2}(x,m)y\cdot y \leq M_T |y|^2,\\
        \nu_T |y|^2 \leq D_{\xi}^2 &\dfrac{\partial h}{\partial m}(x,m)(\xi)y\cdot y \leq M_T|y|^2,\\
        D_{\xi}&\dfrac{\partial ^2 h}{\partial x\partial m}(x,m)(\xi) = 0,\\
        D_{\xi_1}D_{\xi_2}&\dfrac{\partial^2 h}{\partial m^2}(x,m)(\xi_1,\xi_2) = 0.
    \end{align}
    \item [(A4)] For the matrices, we have
    \begin{align}
        |A|<M|BN^{-1}B^*|,\text{ with }|\cdot| \text{ the matrix $2$-norm.}
    \end{align}
\end{enumerate}
The set of our feasible control is $L^2(t,T;L^2_m(\mathbb{R}^n;\mathbb{R}^n))$, i.e., $v_{\cdot,m,t}(\cdot)\in L^2(t,T;L^2_m(\mathbb{R}^n;\mathbb{R}^n))$ if and only if
\begin{align*}
    \int_t^T\int_{\mathbb{R}^n}|v_{x,m,t}(s)|^2dm(x)ds<\infty.
\end{align*}
To each $v_{\cdot,m,t}(\cdot)\in L^2(t,T;L^2_m(\mathbb{R}^n;\mathbb{R}^n))$ and $x\in\mathbb{R}^n$ we associate the state
\begin{align}
    x_{x,m,t}(s;v) := x+\int_t^s \Big[ Ax_{x,m,t}(\tau;v)+Bv_{x,m,t}(\tau)\Big]d\tau.
\end{align}
Note that $x_{\cdot,m,t}(\cdot)\in L^2(t,T;L^2_m(\mathbb{R}^n;\mathbb{R}^n))$. We define the objective functional on $L^2(t,T;L^2_m(\mathbb{R}^n;\mathbb{R}^n))$ by
\begin{align}
    J_{m,t}(v) :=& \int_t^T  F(x_{\cdot,m,t}(s;v)_\# m)ds + F_T(x_{\cdot,m,T}(s;v)_\# m)\\
    &+\dfrac{1}{2}\int_{t}^{T}\int_{\mathbb{R}^n}v_{x,m,t}^*(\tau)Nv_{x,m,t}(\tau)dm(x)d\tau \nonumber.
\end{align}
Thus the value function is 
\begin{align}
    V(m,t) := \inf_{v\in L^2(t,T;L^2_m(\mathbb{R}^n;\mathbb{R}^n))} J_{m,t}(v).
\end{align}
\subsection{THE HILBERT SPACE $\mathcal{H}_m$ AND THE PUSH-FORWARD MAP}
We proceed as our previous work  \cite{stochasticv2}.
\subsubsection{SETTINGS}
Fix $m \in \mathcal{P}_2(\mathbb{R}^n)$, we define $\mathcal{H}_m:= L^2_m(\mathbb{R}^n;\mathbb{R}^n)$, the set of all measurable vector field $\Phi$ such that $\int_{\mathbb{R}^n}|\Phi(x)|^2dm(x)<\infty$. We equip $\mathcal{H}_m$ with the inner product
\begin{align}
    \langle X,Y\rangle_{\mathcal{H}_m}:= \int_{\mathbb{R}^n} X(x)\cdot Y(x)dm(x).
\end{align}
Write the corresponding norm as $\|X\|_{\mathcal{H}_m} = \sqrt{\langle X,X\rangle_{\mathcal{H}_m}}$.
\begin{definition}
For $m\in\mathcal{P}_2$, $X\in \mathcal{H}_m$, define $X\otimes m \in \mathcal{P}_2$ as follow: for all $\phi:\mathbb{R}^n\to\mathbb{R}$ such that $x\mapsto \dfrac{|\phi(x)|}{1+|x|^2}$ is bounded, define
\begin{align}
    \int \phi(x)d(X\otimes m)(x) := \int \phi(X(x))dm(x).
\end{align}
\end{definition}
\begin{rem}
This actually is the push-forward map as we are working on the deterministic case. We write as $X\otimes m$ to align with our treatment of the stochastic case in \cite{stochasticv2}. 
\end{rem}
\noindent We recall several useful properties from \cite{stochasticv2}.
\begin{prop}
\label{property_push_forward}
We have the following properties:
\begin{enumerate}
    \item Let $X$, $Y\in \mathcal{H}_m$, and suppose $X\circ Y\in \mathcal{H}_m$. Then $(X\circ Y)\otimes m = X\otimes (Y\otimes m)$.
    \item If $X(x) = x$ is the identity map, then $X\otimes m =m$.
    \item Let $X\in\mathcal{H}_m$, denote the space $L_X^2(t,T;\mathcal{H}_m)$ to be the set of all processes in $L^2(t,T;\mathcal{H}_m)$ that is adapted to $\sigma(X)$. There exists a natural linear isometry between $L^2_X(t,T;\mathcal{H}_m)$ and $L^2(t,T;\mathcal{H}_{X\otimes m})$.
\end{enumerate}
\end{prop}
\begin{proof}
Please refer to \cite{stochasticv2} Section 2 and Section 3.
\end{proof}
\subsubsection{EXTENDING THE DOMAIN OF FUNCTIONS TO $\mathcal{H}_m$}
\label{extending_functions}
The proofs in this section is standard, we therefore omit unless specified. Readers may refer to \cite{stochasticv2} Section 2. Let $F:\mathcal{P}_2(\mathbb{R}^n)\to \mathbb{R}$, we extend $F$ to be a function on $\mathcal{H}_m$ by $X\mapsto F(X\otimes m)$, $\forall X\in \mathcal{H}_m$. When the domain is $\mathcal{H}_m$, we can talk about G\^ateaux derivative. We actually have the following relation between the G\^ateaux derivative on $\mathcal{H}_m$ and its functional derivative:
\begin{prop}
\label{first_order_gateaux}
Let $F:\mathcal{P}_2(\mathbb{R}^n)\mapsto\mathbb{R}$ have a functional derivative $\frac{dF}{dm}$, and $x\mapsto\frac{dF}{dm}(m,x)$ is continuously differentiable in $\mathbb{R}^n$. Assume that $D\frac{dF}{dm}(m,x)$ is continuous in both $m$ and $x$, and
\begin{align}
    \Bigg| D\dfrac{dF}{dm}(m)(x)\Bigg| \leq c(m)(1+|x|)
\end{align}
for some constant $c(m)$ depending only on $m$. Denote the G\^ateaux derivative as $D_X F(X\otimes m)$, we have
\begin{align}
    D_X F(X\otimes m) = D\frac{dF}{dm}(X\otimes m)(X(\cdot)).
\end{align}
\end{prop}
\noindent We now look at the second order G\^ateaux derivative, denoted as $D^2_X F(X\otimes m)$, note that $D^2_X F(X\otimes m)$ is a bounded linear operator from $\mathcal{H}_m$ to $\mathcal{H}_m$. 
\begin{prop}
In addition to the assumptions in Proposition \ref{first_order_gateaux}, let $F$ has a second order functional derivative $\frac{d^2F}{dm^2}(m)(x_1,x_2)$, assume also $D^2\frac{dF}{dm}(m)(x)$, $D_1\frac{d^2F}{dm^2}(m)(x_1,x_2)$, $D_2\frac{d^2F}{dm^2}(m)(x_1,x_2)$ and \\$D_1D_2\frac{d^2F}{dm^2}(m)(x_1,x_2)$ exist and are continuous, such that
\begin{align}
    \left| D^2\dfrac{dF}{dm}(m)(x) \right| &\leq d(m),\\
    \left| D_1D_2\dfrac{d^2F}{dm^2}(m)(x_1,x_2)\right| &\leq d'(m),
\end{align}
where $d$, $d'$ are constants depending on $m$ only, and $|\cdot|$ is the matrix $2$-norm. Then we have:
\begin{align}
    D^2_X F(X\otimes m)Y(x) = D^2\dfrac{dF}{dm}(X\otimes m)(X(x))Y(x) + \int_{\mathbb{R}^n}D_1D_2\dfrac{d^2F}{dm^2}(X\otimes m)(X(x),X(x'))Y(x')dm(x').
\end{align}
\end{prop}
\noindent Besides, we can view $F(X\otimes m)$ as $m\mapsto F(X \otimes m)$, in this case, we can talk about differentiation with respect to $m$, denote it as $\dfrac{\partial F}{\partial m}$. The following relation between $\dfrac{\partial F}{\partial m}$ and $\dfrac{dF}{dm}$ holds. 
\begin{prop}
Let $F:\mathcal{P}_2(\mathbb{R}^n)\mapsto\mathbb{R}^n$ have a functional derivative and fix $X\in \mathcal{H}_m$. We have
\begin{align}
    \dfrac{\partial F}{\partial m}(X\otimes m)(x) = \frac{dF}{dm}(X\otimes m)(X(x)).
\end{align}
\end{prop}
\noindent Now let $A:\mathbb{R}^n \to \mathbb{R}^n$, we extend it as $\forall X\in \mathcal{H}_m$, $X \mapsto A(X)\in \mathcal{H}_m$,
\begin{align}
    A(X)(x) = A(X(x)).
\end{align}
It is trivial to see that if $A^{-1}$ exists in $\mathbb{R}^n$, then $A^{-1}(X)(x) = A^{-1}(X(x))$ is the inverse of $A$ in $\mathcal{H}_m$. So is the transpose of $A$, if $A$ is a matrix. Again, we can talk about its G\^ateaux derivative. 
\begin{prop}
Let $A$ to be continously differentiable. Denote its derivative to be $dA$. Assume that there exists $k$ such that $|dA(x)|\leq k$ for all $x\in \mathbb{R}^n$, where $|\cdot|$ is the matrix $2$-norm. Then for all $X, Y \in \mathcal{H}_m$, we have
\begin{align}
    D_XA(X)Y(x) = dA(X(x))Y(x).
\end{align}
\end{prop}
\begin{proof}
Let $X,Y, H\in\mathcal{H}_m$, then
\begin{align*}
    &\dfrac{1}{\epsilon} \Big\langle A(X+\epsilon Y) - A(X),H \Big\rangle_{\mathcal{H}_m}\\
    =&\dfrac{1}{\epsilon}\int_{\mathbb{R}^n}\Big[A(X(\xi)+\epsilon Y(\xi)) - A(X(\xi))\Big]\cdot H(\xi) dm(\xi)\\
    =&\int_{\mathbb{R}^n}\int_0^1 dA(X(\xi)+\theta \epsilon Y(\xi))Y(\xi)\cdot H(\xi) d\theta dm(\xi)\\
    \to &\int_{\mathbb{R}^n} dA(X(\xi))Y(\xi)\cdot H(\xi) dm(\xi) = \Big\langle dA(X(\cdot))Y(\cdot),H\Big\rangle_{\mathcal{H}_m}.
\end{align*}
\end{proof}
\begin{prop}
Let $A$ be twice continuously differentiable. Denote its second derivative to be $d^2A$. Note that $d^2A(x)(a,b) \in \mathbb{R}^n$, and $d^2A(x)(a,b) = d^2A(x)(b,a)$. Assume that there exists $k(x)$ such that $\forall a, b \in\mathbb{R}^n$, $|d^2A(x)(a,b)|\leq k(x)$, then we have
\begin{align}
    d^2A(X)(Y,W)(x) = d^2A(X(x))(Y(x),W(x)).
\end{align}
\end{prop}
\begin{proof}
Let $X, Y, W, H \in\mathcal{H}_m$, then
\begin{align*}
    &\dfrac{1}{\epsilon}\Big\langle D_XA(X+W)Y - D_XA(X)Y, H\Big\rangle_{\mathcal{H}_m}\\
    =&\dfrac{1}{\epsilon}\int_{\mathbb{R}^n}\Big[dA(X(\xi)+\epsilon W(\xi))Y(\xi) -dA(X(\xi))Y(\xi) \Big]\cdot H(\xi) dm(\xi)\\
    =&\int_{\mathbb{R}^n}\int_0^1 d^2A(X(\xi)+\theta\epsilon W(\xi))(Y(\xi),W(\xi))\cdot H(\xi) d\theta dm(\xi)\\
    \to& \int_{\mathbb{R}^n}d^2A(X(\xi))(Y(\xi),W(\xi))\cdot H(\xi)dm(\xi).
\end{align*}
\end{proof}
\subsection{CONTROL PROBLEM IN THE HILBERT SPACE $\mathcal{H}_m$}
Recall the definitions of $A$, $B$, $N$, $F$, $F_T$ in Section \ref{mean_field_type_control}. Extend the functions as in Section \ref{extending_functions}. We assume (A1), (A2), (A3) and (A4). It is not hard to derive (\ref{eq:1-8}), (\ref{eq:1-7}) and (\ref{eq:4-7}) from the assumptions. Note that in our case, $b =0$.\\
Now fix $X\in\mathcal{H}_m$ as our initial data. For given $v_{Xt} \in L^2_X(t,T;\mathcal{H}_m)$ (subscript $X$ and $t$ to address the measurability and starting time), consider the dynamics:
\begin{align}
\label{infinite_dimensional_dynamics}
    X(s) = X + \int_t^s \Big[ AX(\tau) + Bv_{Xt}(\tau)\Big] d\tau. 
\end{align}
Denote the process as $X_{Xt}(s) = X_{Xt}(s;v_{Xt})$. Define the cost functional:
\begin{equation}
\label{Objective_functional}
    J_{Xt}(v_{Xt}) := \int_t^T  F(X_{Xt}(s)\otimes m)ds + F_T(X_{Xt}(T)\otimes m)+\dfrac{1}{2}\int_{t}^{T}\langle v_{Xt}(\tau),Nv_{Xt}(\tau)\rangle_{\mathcal{H}_m}d\tau, 
\end{equation}
and the value function is
\begin{align}
    V(X,t) := \inf_{v_{Xt}\in L^2_X(t,T;\mathcal{H}_m)} J_{Xt}(v_{Xt}).
\end{align}
This is in the form of our concerned model in Section 2, with the Hilbert space being $\mathcal{H}_m$.\\
While (\ref{infinite_dimensional_dynamics}) is infinite dimensional, there is a finite dimensional view point of it. For $v_{Xt}\in L_X^2(t,T;\mathcal{H}_m)$,  by Proposition \ref{property_push_forward}, let $\tilde{v}\in L^2(t,T;\mathcal{H}_{X\otimes m})$ be the representative of $v_{Xt}$. Consider 
\begin{align}
    x(s) = x+\int_t^s \Big[ Ax(\tau)+B\tilde{v}(\tau,x)\Big]d\tau.
\end{align}
Denote the solution to be $x(s;x,\tilde{v}(\cdot,x))$. Then we have
\begin{align*}
    X_{Xt}(s;v_{Xt})(x) = x(s;X(x),\tilde{v}(\cdot,X(x))).
\end{align*}
We introduce the notation $X_{xt}(\cdot)$ with a lowercase letter for $x$ to mean $x(\cdot;x,\tilde{v}(\cdot,x))$, and $v_{\cdot t}(s)$ to mean $\tilde{v}(s,\cdot)$. From above we can conclude that the law of $X_{Xt}(s;v_{Xt}(\cdot))$ is $x(s;\cdot,\tilde{v}(\cdot,\cdot))\otimes(X\otimes m)$. Hence the cost functional (\ref{Objective_functional}) can be written as
\begin{align}
    J_{Xt}(v_{Xt}) = &\int_t^T  F(X_{Xt}(s)\otimes m)ds + F_T(X_{Xt}(T)\otimes m)+\dfrac{1}{2}\int_{t}^{T}\langle v_{Xt}(\tau),Nv_{Xt}(\tau)\rangle_{\mathcal{H}_m}d\tau\\
    = &\int_t^T F(x(s;\cdot,\tilde{v}(\cdot,\cdot))\otimes(X\otimes m))ds + F_T(x(T;\cdot,\tilde{v}(\cdot,\cdot))\otimes(X\otimes m))\nonumber\\
    &+\dfrac{1}{2}\int_{t}^{T}\langle v_{Xt}(\tau),Nv_{Xt}(\tau)\rangle_{\mathcal{H}_m}d\tau\nonumber\\
    =: & J_{X\otimes m,t}\nonumber,
\end{align}
that means $J$ depends on $X$ only through $X\otimes m$. Respectively,
\begin{align}
    V(X,t) = \inf_{v_{Xt}\in L^2_X(t,T;\mathcal{H}_m)} J_{Xt}(v_{Xt}) = \inf_{v_{Xt}\in L^2_X(t,T;\mathcal{H}_m)} J_{X\otimes m, t}(v_{Xt}) =:V(X\otimes m,t).
\end{align}
\subsection{NECESSARY AND SUFFICIENT CONDITION FOR OPTIMALITY}
Assume (A1), (A2), (A3) and (A4), we conclude from Theorem \ref{theo5-1} that there exists unique optimal control $\hat{v}_{Xt}(s) = -N^{-1}B^*Z_{Xt}(s)$, where $Z_{Xt}(s)$ together with $Y_{Xt}(s)$ are the unique solution of the system
\begin{align}
    Y_{Xt}(s) &= X + \int_t^s \Big[ AY_{Xt}(\tau)-BN^{-1}B^*Z_{Xt}(\tau)\Big]d\tau, \\
    Z_{Xt}(s) &= \int_s^T \Big[ ( AY_{Xt}(\tau))^*Z_{Xt}(\tau) + D_XF(Y_{Xt}(\tau)\otimes m)\Big] + D_XF_T(Y_{Xt}(T)\otimes m).
\end{align}
Again, because $L^2_X(t,T;\mathcal{H}_m)$ is isometric to $L^2(t,T;\mathcal{H}_{X\otimes m})$, there exists $Y_{\xi t}(s)$, $Z_{\xi t}(s)$ such that $Y_{Xt} = Y_{\xi t}|_{\xi = X}$ and $Z_{Xt} = Z_{\xi t}|_{\xi = X}$, $(Y_{\xi t},Z_{\xi t})$ solving
\begin{align}
    Y_{\xi t}(s) & =\xi + \int_t^s \Big[ AY_{\xi t}(\tau)-BN^{-1}B^*Z_{\xi t}(\tau)\Big]d\tau, \\
    Z_{\xi t}(s) &= \int_s^T \Bigg[ (AY_{\xi t}(\tau))^*Z_{\xi t}(\tau) + D\dfrac{dF}{dm}(Y_{\cdot t}(\tau)\otimes (X\otimes m))(Y_{\xi t}(\tau))\Bigg] \\
    &\quad+ D\dfrac{dF_T}{dm}(Y_{\cdot t}(T)\otimes (X\otimes m))(Y_{\xi t}(T)).\nonumber
\end{align}
As $(Y_{\xi t},Z_{\xi t})$ depends on $m$ through $X\otimes m$, we write $(Y_{\xi,X\otimes m, t},Z_{\xi, X\otimes m, t})$. We can write the value function as
\begin{align}
    V(X,t) &= \int_t^T  F(Y_{Xt}(s)\otimes m)ds + F_T(Y_{Xt}(T)\otimes m)+\dfrac{1}{2}\int_{t}^{T}\langle N^{-1}B^*Z_{Xt}(\tau),B^*Z_{Xt}(\tau)\rangle_{\mathcal{H}_m}d\tau\\
    &= \int_t^T F(Y_{\cdot,X\otimes m,t}(s)\otimes (X\otimes m)) ds+ F_T(Y_{\cdot,X\otimes m,t}(T)\otimes (X\otimes m))\nonumber\\
    &\quad+\dfrac{1}{2}\int_{t}^{T}\int_{\mathbb{R}^n} N^{-1}B^*Z_{\xi,X\otimes m,t}(\tau) \cdot B^*Z_{\xi,X\otimes m,t}(\tau)d(X\otimes m)(\xi)d\tau \nonumber\\
   & =V(X\otimes m, t)\nonumber.
\end{align}
In particular, if we choose $X$ to be the identity function, i.e., $X(x) = x$, recall that $X\otimes m = m$, there exists $(Y_{x,m, t},Z_{x, m, t})$ solving
\begin{align}
    Y_{x,m,t}(s) &= x + \int_t^s \Big[ AY_{x,m,t}(\tau)-BN^{-1}B^*Z_{x,m,t}(\tau)\Big]d\tau, \\
    Z_{x,m,t}(s) &= \int_s^T \Bigg[ (AY_{x,m,t}(\tau))^*Z_{x,m,t}(\tau) + D\dfrac{dF}{dm}(Y_{\cdot,m, t}(\tau)\otimes m)(Y_{x,m,t}(\tau))\Bigg] \\
    &\quad+ D\dfrac{dF_T}{dm}(Y_{\cdot,m,t}(T)\otimes (X\otimes m))(Y_{x,m,t}(T)),\nonumber
\end{align}
which is the system of optimality condition for our mean field type control problem in Section \ref{mean_field_type_control}. For the value function, we have
\begin{align}
    V(m,t) = &\int_t^T F(Y_{\cdot, m,t}(s)\otimes m) ds+ F_T(Y_{\cdot,m,t}(T)\otimes m)\\
    &+\dfrac{1}{2}\int_{t}^{T}\int_{\mathbb{R}^n} N^{-1}B^*Z_{x,m,t}(\tau) \cdot B^*Z_{x,m,t}(\tau)dm(x)d\tau. \nonumber
\end{align}
\subsection{PROPERTIES OF THE VALUE FUNCTION}
We give the functional derivative of the value function $V$, and the relation between the solution of the FBSDE and $V$. As the proofs are standard, we omit here and readers may refer to Section 4 of \cite{stochasticv2}.
\begin{prop}
\label{prop_value_function}
Assume (A1), (A2), (A3), (A4). We have the following properties for the value function:
\begin{enumerate}
    \item By Proposition \ref{prop5-1}, we have
    \begin{align}
        D_XV(X\otimes m,t) = Z_{Xt}(t).
    \end{align}
    \item We have
    \begin{align}
        \dfrac{dV}{dm}(m,t)(x) = &\int_t^T \dfrac{dF}{dm}(Y_{\cdot, m,t}(s)\otimes m) (Y_{x,m,s}(s))ds+ \dfrac{dF_T}{dm}(Y_{\cdot,m,t}(T)\otimes m)(Y_{x,m,s}(T))\\
        &+\dfrac{1}{2}\int_{t}^{T} N^{-1}B^*Z_{x,m,t}(\tau) \cdot B^*Z_{x,m,t}(\tau)d\tau. \nonumber
    \end{align}
    \item We have
    \begin{align}
        D\dfrac{d}{dm}V(m,t)(x) &= Z_{x,m,t}(t),\\
        D_X V(X\otimes m,t) &= D\dfrac{d}{dm}V(X\otimes m, t)(X)
    \end{align}
    \item Also, the feedback nature of $Z$ in $Y$, i.e., for any $x\in\mathbb{R}^n$, $\forall s \in [t,T]$, we have 
    \begin{align}
        Z_{x,m,t}(s) = D\dfrac{d}{dm}V(Y_{\cdot,m,t}\otimes m,s)(Y_{x,m,t}(s)),
    \end{align}
    and for any $X\in\mathcal{H}_m$, $\forall s \in [t,T]$,
    \begin{align}
        Z_{Xt}(s) = D_X V(Y_{Xt}(s)\otimes m,s).
    \end{align}
\end{enumerate}
\end{prop}
\subsection{BELLMAN EQUATION}
Assume (A1), (A2), (A3), (A4). By Theorem \ref{theo5-3}, we deduce that for any $T>0$, $V(X\otimes m,t)$ is the unique solution to the following Bellman equation:
\begin{small}
\begin{align}
\begin{cases}
    -\dfrac{\partial V}{\partial t}(X\otimes m, t) - \Big\langle D_X V(X\otimes m,t), AX \Big\rangle_{\mathcal{H}_m}+\dfrac{1}{2}\Big\langle D_XV(X\otimes m,t),BN^{-1}B^* D_X V(X\otimes m,t) \Big\rangle_{\mathcal{H}_m} = F(X\otimes m),\\
    V(X\otimes m) = F_T(X\otimes m)
\end{cases}
\end{align}
\end{small}%
As before, let $X$ be the identity function, together with Proposition \ref{prop_value_function}, we conclude that for any $T>0$, $V(m,t)$ solves the following PDE on the space of probability measures:
\begin{small}
\begin{align}
\begin{cases}
    -\dfrac{\partial V}{\partial t}(m,t) - \displaystyle\int_{\mathbb{R}^n}D\dfrac{dV}{dm}(m,t)(x)\cdot Ax dm(x)+\displaystyle\dfrac{1}{2}\int_{\mathbb{R}^n}D\dfrac{dV}{dm}(m,t)(x)\cdot BN^{-1}B^*D\dfrac{dV}{dm}(m,t)(x) dm(x)= F(m),\\
    V(m,T) = F_T(m).
\end{cases}
\end{align}
\end{small}

\end{document}